\documentclass{mrlart7}
  \def\volno{18}
  \def\yrno{2011}
  \def\issueno{00}
  \def\lpageno{100NN} 
\usepackage{graphicx}
\usepackage{indentfirst}
\usepackage{amsmath,amsfonts,amsthm,amssymb,xypic}
\usepackage{mathrsfs}
\usepackage{amscd}
\usepackage{color}
\usepackage{hyperref}

\def\co{\colon\thinspace}
\DeclareMathAlphabet{\mathsfsl}{OT1}{cmss}{m}{sl}

\newcommand{\bZ}{{\mathbb{Z}}}

\newcommand{\SKh}{\mbox{SKh}}

\newcommand{\bF}{\mathbb{F}}
\newcommand{\bL}{\mathbb{L}}
\newcommand{\cF}{\mathcal{F}}
\newcommand{\cI}{\mathcal{I}}
\newcommand{\CKh}{\mbox{CKh}}
\newcommand{\Wedge}{\Lambda}
\newcommand{\Kh}{\mbox{Kh}}

\newtheorem{thm}{Theorem}[section]
\newtheorem{lem}[thm]{Lemma}
\newtheorem{cor}[thm]{Corollary}
\newtheorem{prop}[thm]{Proposition}

\theoremstyle{definition}
\newtheorem{defn}[thm]{Definition}

\newtheorem{rem}[thm]{Remark}

\begin{document}

\title{Sutured Khovanov homology distinguishes braids from other tangles}

\author{J. Elisenda Grigsby}

\address{Department of Mathematics, Boston College\\
Carney Hall, Chestnut Hill, MA
02467}
\email{grigsbyj@bc.edu}

\author{Yi Ni}

\address{Department of Mathematics, Caltech, MC 253-37\\
 1200 E California Blvd, Pasadena, CA
91125}

\email{yini@caltech.edu}

\begin{abstract}
We show that the sutured Khovanov homology of a balanced tangle in the product sutured manifold $D \times I$ has rank $1$ if and only if the tangle is isotopic to a braid.
\end{abstract}

\date{}
\maketitle

\section{Introduction}

In \cite{KhovJones}, Khovanov constructed a categorification of the Jones polynomial that assigns a bigraded abelian group to each link in $S^3$. Sutured Khovanov homology is a variant of Khovanov's construction that assigns 
\begin{itemize}
	\item to each link $\bL$ in the product sutured manifold $A \times I$ (see Section~\ref{sec:SKhAnnLink}) a triply-graded vector space $\SKh(\bL)$ over $\bF := \bZ/2\bZ$ \cite{APS_KBSM, Roberts}, where $A=S^1\times[0,1]$ and $I=[0,1]$, and
	\item to each balanced, admissible tangle $T$ in the product sutured manifold $D \times I$ (see Section~\ref{sec:SKhTangle}) a bigraded vector space $\SKh(T)$ over $\bF$ \cite{KhovColJones, GWColJones}, where $D=D^2$.
\end{itemize}

Khovanov homology detects the unknot \cite{KM2010} and unlinks \cite{HN, BS}, and the sutured annular Khovanov homology of braid closures detects the trivial braid \cite{BG}. In this note, we prove that  the sutured Khovanov homology of balanced tangles distinguishes braids from other tangles.

\begin{thm}\label{thm:SKhBraid}
Let $T\subset D\times I$ be a balanced, admissible tangle. (See Subsection~\ref{sec:SKhTangle} for the definition.) Then
$\SKh(T)\cong\mathbb F$ if and only if $T$ is isotopic to a braid in $D\times I$.
\end{thm}

Theorem~\ref{thm:SKhBraid} is one of many results about the connection between Floer homology and Khovanov homology, starting with the work of Ozsv\'ath and Szab\'o \cite{OSzBrCov}. This theorem is
an analogue of the fact that sutured Floer homology detects product sutured manifolds \cite{NiFibred,Ju2}, which is also an ingredient in our proof. Other ingredients include a spectral sequence relating sutured Khovanov homology and sutured Floer homology \cite{GWColJones}, Meeks--Scott's theorem on finite group actions on product manifolds \cite{MS},  and Kronheimer--Mrowka's theorem that Khovanov homology is an unknot detector \cite{KM2010}.

Given a link $\bL\subset A\times I$, the {\it wrapping number} of $\bL$ is the minimal geometric intersection number of all links isotopic to $\bL$ with the meridional disk of $A\times I$. Theorem \ref{thm:SKhBraid} combined with the observations in \cite{GWAnnLink} (see Proposition~\ref{prop:Cut}) imply:

\begin{cor}\label{cor:SKhBraid}
Let $\bL\subset A\times I$ be a link with wrapping number $\omega$, then the group
$$\SKh(\bL;\omega)=\bigoplus_{i,j}\SKh^{i}(\bL;j,\omega)$$
is isomorphic to $\mathbb F$ if and only if $\bL$ is isotopic to a closed braid in $A\times I$.
\end{cor}

This corollary is an analogue of the fact that knot Floer homology detects fibered knots.


\section{Preliminaries}

In this section, we will review the basics about 
sutured manifolds \cite{Ga1} and sutured Khovanov homology \cite{APS_KBSM, Roberts, GWColJones, GWAnnLink}.

\begin{defn}
A {\it sutured manifold} $(M,\gamma)$ is a compact, oriented
3--manifold $M$, a set $\gamma\subset \partial M$, and a choice of orientation on each component of $R(\gamma)=\partial M\setminus \mathrm{int}(\gamma)$ such that:
\newline \qquad$\bullet$ $\gamma$ consists of
pairwise disjoint annuli $A(\gamma)$ and tori $T(\gamma)$,
\newline \qquad$\bullet$ if we define $R_+(\gamma)$ (resp., $R_-(\gamma)$) to be the union
of those components of $R(\gamma)$ whose normal vectors point out
of (resp., into) $M$, then each component of $A(\gamma)$ is adjacent to a component of $R_+(\gamma)$ and a component of $R_-(\gamma)$.
\end{defn}

As an example, let $S$ be a compact oriented surface, $M=S\times I$, $\gamma=(\partial S)\times I$,
$R_-(\gamma)=S\times\{0\}, R_+(\gamma)=S\times \{1\}$, then $(M,\gamma)$ is a sutured manifold. In this case we say that $(M,\gamma)$ is a {\it product sutured manifold}.

\begin{defn}\cite[Definition 2.2]{Ju1}
A {\it balanced sutured manifold} is a sutured manifold
$(M,\gamma)$ satisfying\\ (1) $M$ has no closed components.\\
(2) $T(\gamma)=\emptyset$.\\
(3) Every component of $\partial M$ intersects $\gamma$
nontrivially.\\
(4) $\chi(R_+(\gamma))=\chi(R_-(\gamma))$.
\end{defn}

If $(M,\gamma)$ is a balanced, sutured manifold, then $SFH(M,\gamma)$ will denote its {\em sutured Floer homology}, as defined by Juh{\'a}sz in \cite{Ju1}. Whenever $\gamma$ is implicit (e.g., when $M$ is a product), we shall omit it from the notation.

We will be interested in Khovanov-type invariants for certain links and tangles in product sutured manifolds.

\subsection{Sutured Khovanov homology of links in $A \times I$} \label{sec:SKhAnnLink}

Sutured annular Khovanov homology, originally defined in \cite{APS_KBSM}, \cite{Roberts} (see also \cite{GWAnnLink}) associates to an oriented link $\mathbb{L}$ in the product sutured manifold $A \times I$ a triply-graded vector space \[\SKh(\mathbb{L})=\bigoplus_{i,j,k} \SKh^i(\mathbb{L};j,k),\] which is an invariant of the oriented isotopy class of $\mathbb{L}\subset A\times I$.

To define it, one chooses a diagram $\mathcal{D}_\bL$ of $\mathbb{L}$ on $A\times \left\{\frac{1}{2}\right\}$. By filling in one boundary component of $A \times \{\frac{1}{2}\}$ with a disk marked with a basepoint $X$ at its center and the other boundary component with a disk marked with a basepoint at its center, one obtains a diagram on $S^2-\{X,O\}$. Ignoring the $X$ basepoint yields a diagram on $\mathbb R^2 = S^2 - \{O\}$ from which the ordinary bigraded Khovanov chain complex 
\[\CKh(\mathcal{D}_\bL) :=\bigoplus_{i,j} \CKh^{i}(\mathcal{D}_\bL;j)\] can be constructed from a cube of resolutions.  Here, $i$ and $j$ are the homological and quantum gradings, respectively. The basepoint $X$ gives rise to a filtration on $\CKh(\mathcal{D}_\bL)$, and $\SKh(\bL)$ is the homology of the associated graded object. 

To define this filtration,  choose an oriented arc from $X$ to $O$ missing all crossings of $\mathcal{D}_\bL$. As described in \cite[Sec. 4.2]{GWGradings}, the generators of $\CKh(\mathcal{D}_\bL)$ are in one-to-one correspondence with {\em oriented} resolutions, where the counterclockwise orientation on each circle corresponds to the generator $v_+$. The ``$k$" grading of an oriented resolution is defined to be the algebraic intersection number of this resolution with our oriented arc. Roberts proves (\cite[Lem. 1]{Roberts}) that the Khovanov differential does not increase this extra grading.

One therefore obtains a bounded filtration,
\[ 0 \subseteq \ldots \subseteq \cF_{n-1}(\mathcal{D}_\bL) \subseteq \cF_{n}(\mathcal{D}_\bL) \subseteq \cF_{n+1}(\mathcal{D}_\bL) \subseteq \ldots \subseteq \CKh(\mathcal{D}_\bL),\] where $\mathcal{F}_n(\mathcal{D}_\bL)$ is the subcomplex of $\CKh(\mathcal{D}_\bL)$ generated by oriented resolutions with $k$ grading at most $n$. Let \[\mathcal{F}_n(\mathcal{D}_\bL;j) = \mathcal{F}_n(\mathcal{D}_\bL)\,\cap \,\bigoplus_i \CKh^i(\mathcal{D}_\bL;j).\] The sutured annular Khovanov homology groups of $\bL$ are  defined to be  \[\SKh^i(\bL;j,k) := H^i\left(\frac{\mathcal{F}_{k}(\mathcal{D}_\bL;j)}{\mathcal{F}_{k-1}(\mathcal{D}_\bL;j)}\right).\] 

It is an immediate consequence of the definitions that if $\bL$ has wrapping number $\omega$, then $\SKh^i(\bL;j,k) \cong 0$ for $k \not\in \{-\omega, -(\omega - 2), \ldots, \omega - 2, \omega\}$.

We shall denote by $\Sigma(A \times I, \bL)$ the sutured manifold obtained as the double cover of $A \times I$ branched along $\bL$ (cf. \cite[Rmk. 2.6]{GWAnnLink}), where $\gamma$ is the cover of $(\partial A)\times I$, and $R_+$ (resp., $R_-$) is the cover of $A\times\{1\}$ (resp., $A\times\{0\}$).

\subsection{Sutured Khovanov homology of balanced tangles in $D \times I$} \label{sec:SKhTangle}

A tangle $T$ in the product sutured manifold $(D \times I, \gamma)$ is said to be {\em admissible} if $\partial T \cap \gamma = \emptyset$,  and {\em balanced} if $|T \cap (D \times \{0\})| = |T \cap (D \times \{1\})|$. To make sense of tangle composition (stacking), we will fix an identification of $D$ with the standard unit disk in $\mathbb C$ and assume that $\partial T$ intersects both $D \times \{0\}$ and $D \times \{1\}$ along the real axis.

The sutured Khovanov homology of an admissible, balanced tangle in $D \times I$ was defined by Khovanov in \cite[Sec. 5]{KhovColJones} in the course of constructing a categorification of the reduced $n$--colored Jones polynomial. An elaboration of Khovanov's construction is given in \cite[Sec. 5]{GWColJones}, where it is also related to sutured Floer homology. We briefly recall the main points of the construction here.

Let $T \subset D \times I$ be a balanced, admissible tangle and choose a diagram $\mathcal{D}_T$ of $T$ on $[-1,1] \times I$. Then the sutured Khovanov homology of $T$, $\SKh(T) = \bigoplus_{i,j} \SKh^{i}(T;j),$ is obtained as the homology of the complex, \[\CKh(\mathcal{D}_T) :=\bigoplus_{i,j} \CKh^i(\mathcal{D}_T;j)\] obtained as follows.

Number the $c$ crossings,  and construct a Khovanov-type {\em cube of resolutions} whose vertices are in one-to-one correspondence with elements of $\{0,1\}^c$. Associated to each such $\cI \in \{0,1\}^c$ is a {\em complete resolution} $R_\cI$ with $a_{\cI}$ closed components (circles) $T_1, \ldots, T_{a_\cI}$ and $b_{\cI}$ non-closed components (arcs) $T_{a_{\cI}+1}, \ldots, T_{a_{\cI} + b_{\cI}}$. We say that $R_\cI$ {\em backtracks} if the boundary of at least one of its non-closed components is contained in $[-1,1] \times \{1\}$. We now assign to the corresponding vertex in the cube of resolutions the vector space 

\[V(R_{\cI}) := \left\{\begin{array}{cl}
			0 & \mbox{if $R_{\cI}$ backtracks}\\
			\Wedge^*(Z(R_\cI)) & \mbox{otherwise,}
			\end{array}\right.\]
where \[Z(R_\cI) := \frac{\mbox{Span}_\bF\{[T_1], \ldots, [T_{a_\cI + b_\cI}]\}}{\mbox{Span}_\bF([T_{a_\cI + 1}] , \ldots,[T_{a_\cI + b_\cI}])}\] is the vector space formally generated by the closed components of $R_\cI$, which for convenience we realize as a quotient space of the vector space formally generated by {\em all} components of $R_\cI$.

As in ordinary Khovanov homology, if $\cI'$ is an {\em immediate successor} of $\cI$ in the language of \cite[Sec. 4]{OSzBrCov} and \cite[Sec. 4]{GWColJones}, then one obtains $R_{\cI'}$ from $R_\cI$ by either merging two components $T_i$ and $T_j$ of $R_\cI$ to form a component $T'$ of $R_\cI'$ or splitting a single component $T$ of $R_\cI$ into two components $T_i'$ and $T_j'$ of $R_{\cI'}$, and in both cases leaving all other components unchanged.

With the above understood, we now associate a map \[F_{R_\cI \rightarrow R_{\cI'}}: V(R_\cI) \rightarrow V(R_{\cI'})\] to every pair of immediate successors as follows.

If at least one of $R_\cI$, $R_{\cI'}$ backtracks, we define $F_{R_\cI \rightarrow R_{\cI'}} := 0$.

Otherwise, $R_\cI \rightarrow R_{\cI'}$ is either a merge or split cobordism involving either two closed components or one closed component and one non-backtracking arc.

If $R_{\cI} \rightarrow R_{\cI'}$ is a merge, we define $F_{R_\cI \rightarrow R_{\cI'}}$ to be the composition 
\[\xymatrix{V(R_\cI) \ar[r]^-\pi & \frac{V(R_\cI)}{[T_i] \sim [T_j]} \ar[r]^-\alpha & V(R_\cI')},\] where $\alpha$ is the isomorphism on exterior algebras induced by the isomorphism \[\frac{Z(R_\cI)}{[T_i] \sim [T_j]} \cong Z(R_{\cI'})\] identifying $[T_i] = [T_j]$ with $[T']$. 

If $R_\cI \rightarrow R_{\cI'}$ is a split, we define $F_{R_\cI \rightarrow R_{\cI'}}$ to be the composition 
\[\xymatrix{V(R_\cI) \ar[r]^-{\alpha^{-1}} & \frac{V(R_\cI)}{[T_i']\sim[T_j']} \ar[r]^-\varphi & V(R_{\cI'})},\] where $\varphi(a) := ([T_i'] + [T_j']) \wedge \widetilde{a}$, and $\widetilde{a}$ is any lift of $a$ in $\pi^{-1}(a)$.

The image of $\theta \in V(R_{\cI})$ under the boundary map $\partial$ on the complex is now defined to be \[\partial(\theta) := \sum_{R_{\cI'}} F_{R_{\cI} \rightarrow R_{\cI'}}(\theta),\] where the sum is taken over all immediate successors $\cI'$ to $\cI$. Extend linearly.

\begin{rem} If $T$ is an admissible $(n,n)$ tangle in $D \times I$ and $\mathcal{D}_T$ is a diagram of $T$, then we can alternatively associate to $T$ a left $H^n$--module, $\cF(\mathcal{D}_T)$, as in \cite{KhTangle}, by viewing $T$ as a tangle with $2n$ upper endpoints (cf. \cite[Rmk. 5.9]{GWColJones}). The chain complex $\CKh(\mathcal{D}_T)$ may then be identified with $ \vec{{\bf v}}_- \otimes_{H^n} \cF(\mathcal{D}_T) ,$ where $\vec{{\bf v}}_-$ is the right $H^n$ module constructed as follows. Let $b$ denote the fully-nested crossingless match on $2n$ points; then $\vec{\bf v}_-$ is  the two-sided ideal of the $H^n$ module $\cF(W(b)b)$ corresponding to the generator whose strands are all labeled with a $v_-$. Via the correspondence between {\em oriented resolutions} and Khovanov generators described in the previous section (cf. \cite[Sec. 4.2]{GWGradings}), we may then identify $\CKh(\mathcal{D}_T)$ as the quotient complex obtained from the ordinary Khovanov complex of the closure, $\widehat{\mathcal{D}}_T$, of $\mathcal{D}_T$ by the subcomplex generated by all generators with Roberts' ``k"--grading less than $n$. This has the effect of setting to $0$ any vertex associated to a backtracking resolution and treating the non-backtracking non-closed components of a resolution just as basepointed strands are treated in Khovanov's {\em reduced} theory.
\end{rem}

Comparing the above description with the description of the sutured annular Khovanov invariant in the previous section, we have:

\begin{prop} \label{prop:Cut} {\rm\cite[Thm. 3.1]{GWAnnLink}}
If $\bL \subset A \times I$ is an oriented annular link with wrapping number $\omega$, and $T_\theta$ is the oriented, admissible balanced tangle obtained by decomposing $A \times I$ along a meridional disk $D_\theta$ for which $|\bL \cap D_\theta| = \omega$,
\[\SKh^i(\bL;j,\omega) \cong \SKh^i(T_\theta;j).\]
\end{prop} 

Since all but one resolution of a braid backtracks, we have:

\begin{prop}\label{prop:BraidRank}
If $T\subset D\times I$ is isotopic to a braid, then $SKh(T)\cong\mathbb F$. 
\end{prop}


\section{Proof of the main theorem}

\begin{defn}
A tangle $T\subset D\times I$ is a {\it string link} if it consists of proper arcs, each of which has one end on $D\times\{0\}$ and the other end on $D\times\{1\}$. 
\end{defn}

As a consequence, a string link $T$ contains no closed components, and $T$ does not backtrack.

\begin{lem}\label{lem:StringLink}
Let $T\subset D\times I$ be a balanced, admissible tangle, then $\dim_{\mathbb F}\SKh(T)$ is odd if and only if $T$ is a string link. 
\end{lem}
\begin{proof} 
We observe that if two tangles $T_+,T_-$ differ by a
crossing change, then the corresponding  chain complexes $\CKh(T_+)$ and $\CKh(T_-)$ have the same set of generators, thus the parities of the total dimensions of their homology are the same.

If a tangle $T$ has closed components, after crossing changes we can transform $T$ to a tangle $T'$ with a diagram $\mathcal D'$ containing a
trivial loop. This loop persists in any complete resolution of $\mathcal D'$, so it follows from the construction that the dimension of $\CKh(\mathcal{D}')$ is even, hence  $\dim_{\mathbb F}\SKh(T)$ is even.

If $T$ backtracks, after crossing changes we can transform $T$ to a tangle $T'$ with an arc which can be isotoped rel boundary into $D\times\{0\}$ or $D\times\{1\}$ without crossing other components. We can find a diagram $\mathcal D'$ of $T'$ such that any complete resolution of $\mathcal D'$ backtracks. So $\CKh(\mathcal D')=0$,  and $\dim_{\mathbb F}\SKh(T)$ is even.

If $T$ is a string link, after crossing changes we can transform $T$ to a
braid $B$. By Proposition~\ref{prop:BraidRank}, $\SKh(B)\cong\mathbb F$, so $\dim_{\mathbb F}\SKh(T)$ is odd.
\end{proof}

\begin{defn}
A tangle $T\subset D\times I$ is {\it split}, if there exists a $3$--ball $B\subset D\times I$, such that $L_2=T\cap B$ is a link and $L_2\ne T$. In this case, let $T_1=T-L_2$, then we write $T=T_1\sqcup L_2$. We say $T$ is {\it nonsplit} if it is not split.

A tangle $T\subset D\times I$ is {\it nonprime}, if there exists a $3$--ball $B\subset D\times I$, such that $T_2=T\cap B$ is a $(1,1)$--tangle in $B$, and $T_2$ does not cobound a disk with any arc in $\partial B$. In this case, Let $T_1\subset D\times I$ be the tangle obtained by replacing $T_2$ with a trivial arc in $B$, and let $L_2$ be the link obtained from $T_2$ by connecting the two ends of $T_2$ by an arc in $\partial B$. We denote $T=T_1\#L_2$. 
We say $T$ is {\it prime} if there does not exist such a $B$.
\end{defn}

\begin{lem}\label{lem:DoubIrr}
Let $(M,\gamma)$ be the sutured manifold which is the double branched cover of $D^2\times I$ branched along $T$. Then $M$ is irreducible if and only if $T$ is nonsplit and prime.
\end{lem}
\begin{proof}
The conclusion follows from the Equivariant Sphere Theorem \cite{MSY} by the same argument as in \cite[Proposition~5.1]{HN}.
\end{proof}

\begin{lem}\label{lem:Conn}
If $T=T_1\# L_2$ is a nonprime string link, then
$$
\SKh(T)\cong \SKh(T_1)\otimes Kh_r(L_2).
$$
\end{lem}

In the above, $Kh_r(L_2)$ denotes the reduced Khovanov homology of $L_2$.

\begin{proof}
We choose a diagram $\mathcal{D}_T$ of $T$ realized as the composition of diagrams $\mathcal{D}_{T_1}$ of $T_1$ and $\mathcal{D}_{L_2^{*,n}}$ $L_2^{*,n}$, where $L_2^{*,n}$ is an $(n,n)$ tangle obtained from $L_2$ by removing a neighborhood of a point near the connected sum region and adjoining $n-1$ trivial strands as pictured in Figure~\ref{fig:ConnectSum}.

 \begin{figure}[t]
 \label{fig:ConnectSum}
 \begin{center}
 \includegraphics[width=1in]{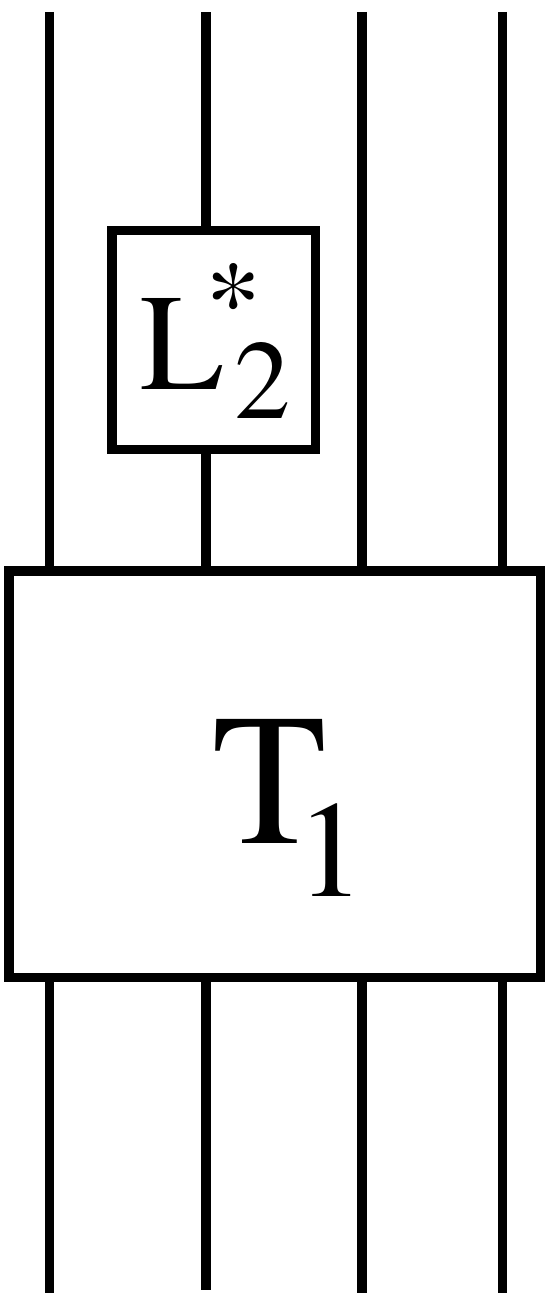}
  \end{center}
 \caption{The tangle $T = T_1 \# L_2$, realized as a composition of $T_1$ and $L_2^{*,n}$.}
 \end{figure}

Now we claim that \[\CKh(\mathcal{D}_T) \cong  \CKh(\mathcal{D}_{T_1}) \otimes_{\bF} \CKh(\mathcal{D}_{L_2^{*,n}}).\] Since $\CKh(\mathcal{D}_{L_2^{*,n}})$ is canonically chain isomorphic to $\CKh(\mathcal{D}_{L_2^{*,1}})$, and the homology of the latter complex is the reduced Khovanov homology of $L_2$ with $\bF$ coefficients, the lemma will then follow from the K{\"u}nneth theorem.

To see the claim, note first that each resolution $R$ of $\mathcal{D}_T$ is obtained by stacking a resolution $R_1$ of $\mathcal{D}_{T_1}$ and $R_2$ of $\mathcal{D}_{L_2^{*,n}}$.

Moreover:
\begin{itemize}
	\item $R$ backtracks iff at least one of $R_1$, $R_2$ backtracks, and
	\item If $R$ does not backtrack, then the number of closed components of $R$ is the sum of the number of closed components of $R_1$ and $R_2$.
\end{itemize}

Hence, the $\bF$--vector space underlying the chain complex $\CKh(\mathcal{D}_T)$ is canonically isomorphic to $\CKh(\mathcal{D}_{T_1}) \otimes_{\bF} \CKh(\mathcal{D}_{L_2^{*,n}})$.

To verify that the boundary map $\partial_{T}$ on $\CKh(\mathcal{D}_{T})$ agrees with the induced boundary map on the tensor product, i.e.: \[\partial_{T} = \partial_{T_1} \otimes \mbox{Id} + \mbox{Id} \otimes \partial_{L_2^{*,n}},\] it is sufficient to verify that the two maps agree on any decomposable generator $\theta= \theta_1 \otimes \theta_2$ of $\CKh(\mathcal{D}_T)$ associated to a resolution $R = (R_1,R_2)$. We may further assume, without loss of generality, that $R$ does not backtrack.

By definition \[\partial_T(\theta) = \sum_{R' = (R_1',R_2')} F_{R \rightarrow R'}(\theta)\] where the sum above is taken over all {\em immediate successors} $R'$ to $R$.

But if $R' = (R_1',R_2')$ is an immediate successor of $R$, then either $R_1'$ is an immediate successor of $R_1$ and $R_2' = R_2$, or vice versa. Assume for definiteness that it is the former, the latter case being analogous.

If $R'$ backtracks, then so does $R_1'$, so: \[F_{R \rightarrow R'}(\theta) = \left(F_{R_1 \rightarrow R_1'} \otimes \mbox{Id}\right)(\theta_1 \otimes \theta_2) = 0.\]

If $R'$ does not backtrack, then the saddle cobordism connecting $R_1$ to  $R_1'$ is a merge (resp., split) connecting either
\begin{itemize}
	\item two closed components of $R_1$ (resp., of $R_1'$), or
	\item one closed and one vertical component of $R_1$ (resp., of $R_1'$).
\end{itemize}

In either case, we see that
\[F_{R \rightarrow R'}(\theta) = \left[F_{R_1 \rightarrow R_1'} \otimes \mbox{Id}\right](\theta_1 \otimes \theta_2).\]

We conclude that 
\begin{eqnarray*}
	\partial_T(\theta) &=& \left[\sum_{R' = (R_1',R_2')} F_{R \rightarrow R'}\right](\theta)\\
	&=& \left[\left(\sum_{R_1'} F_{R_1 \rightarrow R_1'}\right) \otimes \mbox{Id} + \mbox{Id} \otimes \left(\sum_{R_2'} F_{R_2 \rightarrow R_2'}\right)\right] (\theta_1 \otimes \theta_2)\\
	&=& \left[\partial_{T_1} \otimes \mbox{Id} + \mbox{Id} \otimes \partial_{L_2^{*,n}}\right] (\theta_1 \otimes \theta_2),
\end{eqnarray*}
as desired.
\end{proof}

\begin{prop}\label{prop:ProdCov}
Suppose that $T\subset D^2\times I$ is a balanced, admissible tangle. If the double branched cover of $D^2\times I$ branched along $T$ is a product sutured manifold, then $T$ is isotopic to a braid.
\end{prop}
\begin{proof}
Let $\pi\co F\times I\to D^2\times I$ be the double branched covering map, then the nontrivial deck transformation $\rho$ is an involution on $F\times I$ that preserves $F\times\partial I$ setwise.
By Meeks--Scott \cite[Theorem~8.1]{MS}, $\rho$ is conjugate to a map preserving the product structure.\footnote{A homeomorphism of $X\times Y$ {\it preserves the product structure} if it is the product of homeomorphisms of $X$ and $Y$.} In particular, $\pi^{-1}(T)$, being the set of fixed points of $\rho$, is homeomorphic to $P\times I\subset F\times I$ for some finite set $P\subset F$, via a homeomorphism of $F\times I$ which preserves $F\times\partial I$. It follows that $T$ is isotopic to a braid.
\end{proof}

\begin{prop}\label{prop:KM}
A knot $K\subset S^3$ is the unknot if and only if $Kh_r(K)\cong\mathbb F$.
\end{prop}
\begin{proof}
This result is essentially a theorem of Kronheimer and Mrowka \cite{KM2010}. The original theorem of Kronheimer and Mrowka states that $K$ is the unknot if and only if $Kh_r(K;\mathbb Z)\cong\mathbb Z$, where the coefficients ring is $\mathbb Z$ while ours is $\mathbb F$. However, the version with $\mathbb F$ coefficients easily follows from Kronheimer and Mrowka's argument. As shown in \cite[Corollary~1.3]{KM2010},
$$\mathrm{rank\:}Kh_r(K;\mathbb Z)\ge\mathrm{rank\:}I^{\natural}(K).$$
Kronheimer and Mrowka proved that $\mathrm{rank\:}I^{\natural}(K)>1$ when $K$ is nontrivial. (See the paragraph after \cite[Corollary~1.3]{KM2010}.) So $\mathrm{rank\:}Kh_r(K;\mathbb Z)>1$ when $K$ is nontrivial. It follows from the universal coefficients theorem that $\dim_{\mathbb F}Kh_r(K;\mathbb F)>1$ when $K$ is nontrivial.
\end{proof}

\begin{proof}[Proof of Theorem~\ref{thm:SKhBraid}]
By Lemma~\ref{lem:StringLink}, if $\SKh(T)\cong\mathbb F$, then $T$ is a string link. In particular, $T$ has no closed components, hence $T$ must be nonsplit.

Since $T$ has no closed components, if $T$ is nonprime it must be the connected sum of a tangle with a knot (rather than a link). Suppose that
$T=T_1\# K_2$, where $K_2$ is a knot. Then it follows from Lemma~\ref{lem:Conn} that $Kh_r(K_2)\cong\mathbb F$. 
Using Proposition~\ref{prop:KM}, we conclude that $K_2$ is the unknot. Hence $T$ is prime.

Since $T$ is nonsplit and prime, Lemma~\ref{lem:DoubIrr} implies that $\Sigma(D\times I,T)$ is irreducible.
Suppose that $\SKh(T)\cong\mathbb F$. By \cite[Proposition~5.20]{GWColJones}, there is a spectral sequence whose $E^2$ term is $SKh(T)$ and whose $E^{\infty}$ term is the sutured Floer homology group $SFH(\Sigma(D\times I,T))$. Hence $SFH(\Sigma(D\times I,T))\cong\mathbb F$. In \cite{NiFibred,Ju2}, it is shown that an irreducible balanced sutured manifold $(M,\gamma)$ is a product sutured manifold if and only if $SFH(M,\gamma)\cong\mathbb F$. Hence $\Sigma(D\times I,T)$ is a product sutured manifold. Proposition~\ref{prop:ProdCov} then implies that $T$ is isotopic to a braid.
\end{proof}

\section*{Acknowledgements}
The first author was partially supported by NSF grant numbers DMS-0905848 and CAREER DMS-1151671. The second author was partially supported by NSF grant number DMS-1103976 and an Alfred P. Sloan Research Fellowship. We thank the anonymous referee for a number of valuable suggestions.

\bibliographystyle{mrl}

\bibliography{SKhBraidsRef}

\end{document}